\documentclass[graybox,envcountchap,sectrefs, smallextended]{article}
\usepackage{amsfonts, amssymb, tikz, amsmath,stmaryrd}
\usepackage{amsthm, bbm}
\usetikzlibrary{matrix, arrows}
\usepackage{mathrsfs}
\usepackage{mathabx,epsfig}

\newtheorem{thm}{Theorem}[section]
\newtheorem{cor}[thm]{Corollary}
\newtheorem{prop}[thm]{Proposition}
\newtheorem{lem}[thm]{Lemma}
\newtheorem{defn}[thm]{Definition}

\theoremstyle{remark}
\newtheorem{rem}[thm]{Remark}

\newcommand{\Q}{\mathbb{Q}}
\newcommand{\C}{\mathbb{C}}
\newcommand{\N}{\mathbb{N}}
\newcommand{\Z}{\mathbb{Z}}
\newcommand{\K}{\mathbbm{k}}
\newcommand{\F}{\mathbb{F}}

\newcommand{\g}{\mathfrak{g}}
\renewcommand{\r}{{\mathfrak{r}}}

\renewcommand{\sl}{\mathfrak{sl}}

\newcommand{\m}{\mathfrak{m}}

\newcommand{\z}{\mathfrak{z}}

\newcommand{\n}{\mathfrak{n}}

\renewcommand{\O}{\mathcal{O}}

\newcommand{\B}{\mathcal{B}}

\newcommand{\D}{\mathcal{D}}
\renewcommand{\P}{\mathcal{P}}
\newcommand{\Wh}{\mathcal{W}}

\newcommand{\G}{\mathcal{G}}

\newcommand{\ad}{\operatorname{ad}}
\newcommand{\Lie}{\operatorname{Lie}}
\newcommand{\op}{{\operatorname{op}}}
\newcommand{\End}{\operatorname{End}}

\newcommand{\md}{\operatorname{-mod}}
\newcommand{\Mat}{\operatorname{Mat}}
\newcommand{\Spec}{\operatorname{Specm}}
\newcommand{\gr}{\operatorname{gr}}
\newcommand{\Coker}{\operatorname{Coker}}
\newcommand{\Ker}{\operatorname{Ker}}
\newcommand{\Ann}{\operatorname{Ann}}

\newcommand{\Hom}{\operatorname{Hom}}
\newcommand{\Ad}{\operatorname{Ad}}

\newcommand{\Aut}{{\operatorname{Aut}}}
\newcommand{\Id}{{\operatorname{Id}}}

\newcommand{\oTheta}{\overline{\Theta}}
\newcommand{\wU}{\widehat{U}}
\newcommand{\wI}{\widehat{I}}
\renewcommand{\a}{{\textbf{a}}}
\renewcommand{\b}{{\textbf{b}}}
\newcommand{\mpg}{\m^{\perp_\g}}

\begin{document}
\title{A Morita theorem for modular finite $W$-algebras}
\author{L. W. Topley\footnote{\sc Address: Torre Archimede, Via Trieste 63, 35121-I Padova, Italy \newline Email: lewis@math.unipd.it \newline Telephone: (0039) 329 0947625}}
\date{\small{\today}}
\maketitle

\newcounter{parno}
\renewcommand{\theparno}{{\bf \thesection.\arabic{parno}}}
\newcommand{\p}{\smallskip \refstepcounter{parno}\noindent \theparno \space }
\newcommand{\clear}{\setcounter{parno}{0}}

\newcounter{parnoa}
\renewcommand{\theparnoa}{{\bf \thesubsection.\arabic{parnoa}}}

\abstract{We consider the Lie algebra $\g$ of a simple, simply connected algebraic group over a field of large positive characteristic. For each nilpotent orbit $\O \subseteq \g$ we choose a representative $e\in \O$ and attach a certain filtered, associative algebra $\wU(\g,e)$ known as a finite $W$-algebra, defined to be the opposite endomorphism ring of the generalised Gelfand-Graev module associated to $(\g, e)$. This is shown to be Morita equivalent to a certain central reduction of the enveloping algebra of $U(\g)$. The result may be seen as a modular version of Skryabin's equivalence.

}


\section{Introduction}\label{intro}
\setcounter{parno}{0}

\p A complete classification of irreducible representations of the Lie algebras associated to reductive algebraic groups in positive characteristic $p$ has
remained elusive since the commencement of the theory. It is known that all such representations are finite dimensional, and a fundamental observation
of Kac and Weisfeiler shows that they all factor through certain finite dimensional quotients $U(\g) \twoheadrightarrow U_\chi(\g)$ called reduced enveloping
algebras. Twenty years ago Premet obtained lower bounds on the dimensions of $U_\chi(\g)$-modules \cite{Pr95b} under mild assumptions on $p$, confirming
a long-standing conjecture of \cite{KW71}. These bounds soon led him to define two families of algebras \cite{Pr02}: the restricted finite $W$-algebras and the
ordinary finite $W$-algebras over $\C$. It was immediately demonstrated that the reduced enveloping algebra is a matrix algebra over the restricted finite
$W$-algebra, whilst the ordinary finite $W$-algebra has allowed wonderful progress to be made in the representation theory of complex semisimple Lie
algebras (see \cite{Lo10a} for example). In this paper we consider the finite $W$-algebra $\wU(\g,e)$, which is defined as the opposite endomorphism ring 
of the mod $p$ reduction of the ordinary generalised Gelfand--Graev module. We prove an unrestricted version of the matrix isomorphism theorem mentioned above, which
should be seen as a modular version of Skryabin's equivalence, and which (we hope) shall pave the way for similar progress to be made in the representation
theory of modular Lie algebras.

\p\label{settup} Let $G$ be a simple, simply connected algebraic group over an algebraically closed field $\K$ of very good characteristic $p > 0$, and let $\g$ be its Lie algebra.
For $\chi \in \g^*$ we make the notation
\begin{eqnarray*}
& & d(\chi) := \frac{1}{2}\dim \Ad^*(G)\chi;\\
& & D(\chi) := p^{d(\chi)}.
\end{eqnarray*}
Since $\g \cong \g^*$ as $G$-modules (see \cite[Lemma~I.5.3]{SS70}, for example) we may transfer the notions of nilpotent and semisimple to elements of $\g^*$ and in order to understand the representation theory of
$\g$ it is sufficient to study representations of reduced enveloping algebras $U_\chi(\g)$ where $\chi \in\g^*$ is nilpotent (see Section~\ref{reas} for more detail). 
The nilpotent orbits are classified by the Bala--Carter theorem and we begin by choosing a label {\sf X} corresponding to a nilpotent coadjoint orbit $\O$.
Following the work of Premet (\cite{Pr07}, \cite{Pr10a}) we choose an element $\chi\in \g^*$ associated to a nilpotent element $e\in \g$ and attach a nilpotent algebra
$\m \subseteq \g$ with nice properties. This algebra is obtained by reduction modulo $p$, which can only be carried out
provided $p$ is larger than an unknown bound, \emph{for the rest of the section
we assume that $p$ is larger than the number $\Pi({\sf X})$ introduced in Definition~\ref{pidef}}. Note that this bound depends upon $\Phi, {\sf X}$ and a certain choice of nilpotent element
in the corresponding complex simple Lie algebra. Later on (Remark~\ref{reeem}) we shall explain why it is not yet theoretically possible to calculate $\Pi({\sf X})$, or even to bound it.
An approach which relaxes this bound will be considered in a forthcoming article \cite{GT16} joint with Simon Goodwin.

\p The linear form $\chi$ restricts to a character on $\m$, and the corresponding one dimensional module $\K_\chi$ has $p$-character $\chi|_{\m}$. The induced module
$Q_\chi^{[p]}:=U_\chi(\g)\bigotimes_{U_\chi(\m)}\K_\chi$ is known as the restricted generalised Gelfand-Graev module, whilst
$$U^{[p]}(\g,e) := \End_\g(Q_\chi^{[p]})^\op$$ is called the restricted finite $W$-algebra. Premet's Morita theorem \cite[Proposition~4.1]{Pr07}
states that $U_\chi(\g)$ is isomorphic to ${D(\chi)}$ copies of $Q_\chi^{[p]}$ as a left $\g$-module, and isomorphic to $\Mat_{{D(\chi)}} U^{[p]}(\g,e)$ as algebras.
The purpose of this article is to prove analogues of this theorem for the unrestricted GGG module and FWA. We shall shall call the following
pair the ``generalised Gelfand-Graev module'' and ``finite $W$-algebra'' respectively
\begin{eqnarray*}
Q_\chi := U(\g) \bigotimes_{U(\m)} \K_\chi ; \\
\wU(\g,e) := \End_\g(Q_\chi)^\op. 
\end{eqnarray*}
These are defined in parallel with the more famous finite $W$-algebra over $\C$, which we shall call the ordinary FWA in this article. See Remark~\ref{discussdefns} for
a comparison of the various definitions of modular finite $W$-algebras which occur in \cite{Pr07} and \cite{Pr10a}.

\p The enveloping algebra $U(\g)$ is a free module of finite rank over its $p$-centre $Z_0(\g)$ and the latter algebra is isomorphic to $S(\g)^p = \{ f^p : f\in S(\g) \}$  which identifies naturally with the coordinate ring $\K[(\g^\ast)^{(1)}]$ on the Frobenius twist of $\g^\ast$. Write $\m^\perp \subseteq \g^\ast$ for the space $\{f \in \g^\ast : f(\m) = 0\}$. Since $(\g^\ast)^{(1)}$ identifies with $\g^*$ as topological spaces we may denote by $I(\m)$ the prime ideal of $Z_0(\g)$ consisting of regular functions vanishing on $\chi + \m^\perp$ and consider the central reduction
$$U_{I(\m)}(\g) := U(\g)/I(\m)U(\g).$$
The following theorem is the main result of this paper. Its consequences and generalisations shall be discussed following the proof which is given in Section~\ref{proofofthe}.
\begin{thm}\label{mainthm}
For $p \geq \Pi({\sf X})$ the following holds:
\begin{enumerate}
\item{$U_{I(\m)}(\g) \cong \bigoplus_{{D(\chi)}} Q_\chi$ as left $U(\g)$-modules}
\item{$U_{I(\m)}(\g) \cong \Mat_{{D(\chi)}} \wU(\g,e)$ as algebras.}
\end{enumerate}
\end{thm}

\begin{rem}\label{rem1}
This theorem is very much in the spirit of Premet's original Morita theorem proven in \cite[Theorem~2.3]{Pr02} and we conjecture our theorem should hold in the same general context of $\xi$-admissible algebras. In particular, let
$\r$ be a restricted Lie algebra over any field of positive characteristic, $\xi \in \r^\ast$ and $\n \subseteq \r$ a $\xi$-admissible algebra in the sense of definition~2.3 in {\it op. cit}. We write $\K_\xi$ for the one dimensional representation of $\n$ afforded by $\xi$.
Let $I$ be the ideal of $Z_0(\r)$ vanishing on $\xi + \n^\perp$. Then we conjecture that $U(\r) \bigotimes_{U(\n)} \K_\xi$ is a projective $U(\r) / IU(\r)$-module and
$$U(\r)/IU(\r) \cong \Mat_{p^{\dim \n }} \End_\r( U(\r) \bigotimes_{U(\n)} \K_\xi)^\op.$$
\end{rem}

\p This paper is organised as follows. In Section~\ref{reas} we recall some classical results which reduce the study of representations of $\g$ to the
study of $U_\chi(\g)$-modules as $\chi$ varies over a set of representatives for the nilpotent coadjoint $G$-orbits. In Section~\ref{redmodp} we
recall the process of reduction modulo $p$ for finite $W$-algebras and recall some of their fundamental properties, first discovered by Premet.
The first paragraph of Section~\ref{generalprops} is pivotal, as it is there that
we explain how to reverse engineer the reduction modulo $p$ process. In other words, we explain how to choose a root system $\Phi$, a Bala--Carter label {\sf X}
corresponding to a nilpotent orbit and a nilpotent element $e$ in the complex Lie algebra corresponding to $\Phi$. From these pieces of data we can
choose a lower bound $\Pi({\sf X})$ on the characteristic $p$ of $\K$ such that for $p > \Pi({\sf X})$ we are able to reduce $e$ into the corresponding Lie algebra $\g$
defined over $\K$, define $\chi$ in the associated coadjoint orbit and construct $\m$ and hence $\wU(\g,e)$ by reduction modulo $p$. At this point we will have placed
our discussion in the context of the main theorem. In Section \ref{generalprops} we consider central reductions of enveloping algebras and their modules
and prove some general facts, which are immediately applied to the GGG modules. In Section~\ref{proofofthe} we supply the proof of the main theorem
and prove a corollary which states that $Q_\chi$ is a projective generator for $U_{I(\m)}(\g)\md$. Some of the most celebrated applications of ordinary finite
$W$-algebras are their applications to the representation theory of complex semisimple Lie algebras (see \cite{Lo10b}, \cite{Pr14} and the references therein).
The most fundamental result of this nature is the celebrated Skryabin equivalence which describes an equivalence between the category of modules for the
 ordinary finite $W$-algebra and a certain full subcategory of modules for the associated enveloping algebra, the category of so-called generalised Whittaker
 modules (first proven in \cite[Appendix]{Pr02}; see (\ref{skryequiv1}) for a summary). Our main theorem may be seen as a modular version of Skryabin's
 equivalence and we make this comparison explicit in Theorem \ref{Skry} of Section \ref{Skryabins}. In Section \ref{pcharofmods} we recall an important
 theorem proven independently by Premet, Losev and Ginzburg which characterises the simple modules which are obtained from Skryabin's equivalence (over $\C)$
 in terms of their associated varieties. We show that in the modular case, the modules obtained from Skryabin's equivalence can be characterised in a very
 similar manner, with decomposition classes of $p$-characters replacing associated varieties (Theorem \ref{invariantsofsimples}). In the last section we
 obtain a vast generalisation of the main theorem: for each ideal $I \unlhd Z_0(\g)$ containing $I(\m)$ we prove that the algebra $U(\g)/IU(\g)$ is a matrix
 algebra of rank $D(\chi)$ over a certain quotient of $\wU(\g,e)$ which we describe fairly explicitly (Theorem \ref{generalise}). Our final result,
 Corollary~\ref{comdiag}, clarifies the relationship between restricted and unrestricted modular finite $W$-algebras.

\p Given the interest surrounding ordinary finite $W$-algebras in the past thirteen years and the numerous breakthroughs regarding ordinary representations of Lie algebras, it is the author's hope that similar progress may be made in the modular case using these finite $W$-algebras. Losev famously constructed the ordinary finite $W$-algebra as as the Fedesov quantisation of the Poisson algebra structure which occurs naturally on the Slodowy slice \cite{Lo10b}. Although the theory of Fedesov quantisation in positive characteristics has only recently started to be developed in \cite{BK08} it is natural to wonder whether there exists a description of $\wU(\g,e)$ comparable to Losev's.

\medskip

\noindent {\bf Acknowledgement.} I would like to thank the anonymous referee for making many helpful suggestions to improve the clarity of this article, and for explaining the simple argument for Theorem~\ref{generalise}. I would also like to thank Vanessa Miemietz and Shaun Stevens at the University of East Anglia for many useful discussions, and Max Nazarov for hosting me at the University of York. The research leading to these results has received funding from the European Commission, Seventh Framework Programme, under Grant Agreement number 600376, as well as grant CPDA125818/12 from the University of Padova.

\section{Reduced enveloping algebras}\label{reas}
\setcounter{parno}{0}

\p The facts we recall in this section are extremely classical; we recommend either \cite{Ja98} or \cite{Ja04b} for a survey of the theory. We continue with the assumptions of the introduction, although \emph{in this current section the characteristic $p$ of $\K$ may be good for $G$, although we now assume that $\g \cong \g^*$ as $G$-modules}. Recall that $G$ is a simple, simply connected algebraic group over $\K$. The fact that $\g = \Lie(G)$ endows $\g$ with a $G$-equivariant restricted structure which we denote by $x \mapsto x^{[p]}$. The enveloping algebra $U(\g)$ is naturally filtered $U(\g) =  \bigcup_{i\geq 0} U_i(\g)$ and the PBW theorem states that the associated graded algebra $\gr U(\g)$ is the symmetric algebra $S(\g)$. The $p$-centre $Z_0(\g) \subseteq U(\g)$ is the central subalgebra generated by expressions $x^p - x^{[p]}$ with $x\in \g$. As $x \mapsto x^{[p]}$ is $G$-equivariant, $Z_0(\g)$ is isomorphic to $S(\g)^p := \{f^p : f \in S(\g)\}$ as $G$-modules which, in turn, is isomorphic to the coordinate ring $\K[(\g^\ast)^{(1)}]$. This isomorphism leads to an explicit description of the maximal spectrum
\begin{eqnarray}\label{oneequat}
\Spec Z_0(\g) \overset{G}{\longleftrightarrow} (\g^\ast)^{(1)}.
\end{eqnarray}
\noindent For $\chi \in \g^*$ the map $x^p - x^{[p]} \mapsto \chi(x)^p$ extends to an algebra homomorphism $Z_0(\g) \rightarrow \K$ and the kernel is denoted $I_\chi$. Now the bijection in (\ref{oneequat}) can be made explicit: to each $\chi \in \g^\ast$ we associate the maximal ideal
\begin{eqnarray}\label{idealIchi}
I_\chi = (x^p - x^{[p]} - \chi(x)^p : x\in \g)\unlhd Z_0(\g)
\end{eqnarray}

\p The reduced enveloping algebra with $p$-character $\chi \in \g^\ast$ is defined $$U_\chi(\g) := U(\g)/ I_\chi U(\g).$$ It has dimension $p^{\dim(\g)}$ and every simple $\g$-module factors through precisely one of the quotients $U(\g) \twoheadrightarrow U_\chi(\g)$. The restricted structure on $\g$ allows us to define the notions of nilpotent and semisimple elements of $\g$, and these definitions play an important role in representation theory. According to \cite[Theorem~2.3.5]{FS88} the usual Jordan decomposition theorem holds in this case. The nilpotent orbits are classified by the Bala--Carter method, later improved in \cite{Po80} and \cite{Pr03}. There are two features of the classification pertinent to our current discussion:
\begin{enumerate}
\item{there are finitely many nilpotent orbits;}
\item{the labels which classify the orbits are the same in every good characteristic. These are referred to as ``Bala--Carter labels'' in this article.}
\end{enumerate}

\p Since the characteristic has been assumed to be very good we can pick an isomorphism $\g \rightarrow \g^\ast$. This allows us to transfer the Jordan decomposition to elements of $\g^\ast$, which subtends the notion of a nilpotent or semisimple $p$-characters. If $\chi = \chi_s + \chi_n$ is the Jordan decomposition then a famous result due to Kac--Weisfeiler \cite[Theorem~2]{KW71} (later refined by Friedlander--Parshall \cite[Theorem 3.2]{FP88} and generalised by Premet \cite[Proposition~2.5]{Pr02}) states that $U_\chi(\g)$ is Morita equivalent to $U_{\chi}(\g_{\chi_s})$. The centraliser $\g_{\chi_s} := \{x \in \g : \ad^\ast(x)\chi_s = 0\}$ is the Lie algebra of a Levi subgroup $L$ of $G$ such that the simple factors of $[L, L]$ satisfy the same hypotheses we have placed on $G$. Using a reduction argument this allows us to assume that the $p$-character is supported on the centre of $\g$. Another easy reduction allows to to suppose that in fact $\chi_s = 0$ (see \cite[B.8, B.9]{Ja04b} for more detail). If we are interested in understanding the representation theory of $\g$ then our remarks thus far indicate that we may focus our attention on reduced enveloping algebras with nilpotent $p$-characters. Furthermore, the adjoint action of $G$ on $U(\g)$ induces isomorphisms $U_\chi(\g) \cong U_{\Ad^\ast(g)\chi}(\g)$ which shows that it suffices to allow $\chi$ to vary over a finite set of representatives of the nilpotent coadjoint orbits.

\section{Reducing the ordinary finite $W$-algebra modulo $p$}\label{redmodp}
\setcounter{parno}{0}

\p The purpose of this section is to recall the details of the reduction modulo $p$ process for
finite $W$-algebras which was used to great effect by Alexander Premet in \cite{Pr07} and \cite{Pr10a}.
Although the setup here is very similar to those two papers, it is identical to neither and so we
we must recall this process in full. The notation of the current section shall be slightly different
to the others since we cannot fix a single choice of prime characteristic throughout.

\p\label{cfwa} Fix throughout this section a reduced, indecomposable, crystallographic root system
$\Phi$. We let $G_0$ be a simple, simply connected algebraic group
over $\C$ with root system $\Phi$ and Lie algebra $\g_0 = \Lie(G_0)$.
Choose a Bala--Carter label {\sf X} associated to the root system.
Choose a Chevalley basis $\B$ for $\g_0$ and let $\g_\Z$ denote the corresponding Chevalley $\Z$-form.
We denote by $\kappa_0 : \g_0\times \g_0 \rightarrow \C$ the Killing form.
For any good prime characteristic $p$ we write $\K_p :=  \overline{\F}_p$ with identity element $1_p$.
Denote by $G_p$ the simple, simply connected algebraic group over $\K_p$ with root system $\Phi$ and by $\g_p$ its Lie algebra,
which we identify with $\g_\Z \bigotimes_\Z \K_p$.
The construction of the ordinary finite $W$-algebra begins with a choice of element in the orbit $\O$ assoiated to {\sf X}.
\begin{lem}\label{goodreps}
For each nilpotent orbit $\O \subseteq \g_0$ we can choose a representative $e$ such that the $G_p$-orbit of
$e \bigotimes_\Z 1$ in $\g_p$ has the same Bala--Carter label as $\O$. Furthermore
$$\dim_\C (G_0 e) = \dim_{\K_p} (G_p e\otimes 1).$$
\end{lem}
\begin{proof}
Combine \cite[Ch. III, 4.29]{SS70} and \cite[Theorem~2.6\&~2.7]{Pr03}.
\end{proof}

\p\label{formandadmiss} For each root system and each nilpotent orbit in $\g_0$ a choice of such $e$ shall remain fixed throughout this section.
Since the Jacobson--Morozov theorem holds over $\g_\Q := \g_\Z \bigotimes_\Z \Q$ we can choose an $\sl_2$-triple $(e,h,f)$ in $\g_\Q$.
The form $\kappa_0$ can be rescaled to a form $(\cdot, \cdot)_0 : \g_0 \times \g_0 \rightarrow \C$ with $(e,f)_0 =1$.
\begin{defn}\label{admissible}
We say that a ring $A$ satisfying $\Z \subseteq A \subseteq \Q$  is ``admissible for ${\sf X}$'' if $A = \Z[S^{-1}]$ where $S$
is a finite set of prime numbers containing:
\begin{itemize}
\item{all primes which are not very good for $\Phi$;}
\item{all prime factors of the determinant of the Gram matrix $\left(\kappa (b_1, b_2) \right)_{b_1, b_2 \in \B}$;}
\item{all of the prime factors of $\kappa_0(e,f)$, which is an integer by $\sl_2$-theory.}
\end{itemize}
We denote by $\pi(A)$ the set of all primes $p \in \Z$ with $p A \neq A$.
\end{defn}
\noindent We pick an admissible ring now (say, the smallest one) although as this section progresses we shall enlarge $A$
at several points. Whenever we consider $\K_p$ in this section we will be assuming that $p \in \pi(A)$ where $A$ is the
admissible ring which we have fixed at that moment.
\begin{lem}
For each admissible ring $A$ and each $p \in \pi(A)$ we have $A/pA \cong \F_p$.
Furthermore, there is a positive integer $\Pi(A)$ such that for all primes $p \geq \Pi(A)$ we have $p \in \pi(A)$.
\end{lem}
\begin{rem}
When an admissible ring is chosen we shall always identify $\overline{A/I_p}$ with $\K_p$ for $p \in \pi(A)$.
\end{rem}
\begin{proof}
By definition $A = \Z[p_1^{-1},...,p_r^{-1}]$ for primes $p_1,...,p_r \in \Z$ and so for
$p \notin\{p_1,...,p_r\}$ we have $A/pA \cong (\Z/p\Z)[p_1^{-1},...,p_r^{-1}] \cong \F_p$ since
$p_1,...,p_r$ are all invertible in $\Z/p\Z$. For the second claim notice that
$\pi(A) = \{\textnormal{primes}\}\setminus \{p_1,...,p_r\}$ and so we may take $\Pi(A) := \max \{p_1,...,p_r\}+1$.
\end{proof}

\p\label{qsnas} We aim to construct the ordinary finite $W$-algebra along with an $A$-form.
This will lead immediately to the mod $p$ reduction of the finite $W$-algebra, for each $p \in \pi(A)$.
The element $h \in \g_\Q$ induces a grading $\g_0 = \bigoplus_{i\in\Z} \g_0(i)$ where
$\g_0(i) := \{x\in \g_0 | \ad(h)x = i x\}$. Since $h\in \g_\Q$ there is a similar decomposition
$\g_\Q = \bigoplus_{i\in \Z} \g_\Q(i)$ where $\g_\Q(i) := \g_\Q \bigcap \g_0(i)$.
Let $\g_A$ be the free $A$-module on $\B$, endowed with the structure of an $A$-Lie algebra. When
$p \in \pi(A)$ we identify the restricted Lie algebra $\g_p$ with $\g_A\bigotimes_A \K_p$.
After enlarging $A$ by inverting a finite set of primes $S_1 \subseteq \Z$ we can assume that
$\g_A(i) := \g_A \bigcap \g_0(i)$ are free $A$-modules for all $i \in \Z$.

\p For $p \in \pi(A)$ we write $e_p := e\bigotimes_{A} 1_p$. Our definition of an admissible ring ensures that
$(\cdot, \cdot)_0$ is $A$-valued on $\g_A$, which allows us to define a form $(\cdot, \cdot)_p$ on $\g_p$
by setting $(b_1\bigotimes_A 1_p , b_2\bigotimes_A 1_p)_p := (b_1, b_2)\bigotimes_A 1_p$ and extending by linearity.
This also means that $\chi := (e, \cdot)_0 \in \g_0^*$ restricts to an element of $\g_A^* = \Hom_A(\g_A, A)$
so we can define $\chi_p := (e_p, \cdot)_p$.

\p\label{symplect} Consider the form $\Psi : \g_0(-1) \times \g_0(-1)\rightarrow \C; x, y \mapsto \chi[x,y]$.
The representation theory of $\sl_2$, along with the equivariance of the form $(\cdot, \cdot)_0$,
shows that $\Psi$ is symplectic so that $\dim \g_0(-1) = 2s$ for some $s$,
and we may choose a Witt basis $z_1,...,z_s, z_1',...,z_s'$ for $\g_0(-1)$ lying in $\g_\Q(-1)$
satisfying $\Psi(z_i, z_j') = \delta_{i,j}$ and $\Psi(z_i, z_j) = \Psi(z_i', z_j') = 0$.
After inverting another finite set of primes $S_2\subseteq \Z$ inside $A$ we may assume that $z_1,...,z_s, z_1',...,z_s'$
lie in $\g_A$ and span $\g_A(-1)$ as a free $A$-module. Make the notation $\g_A(-1)^0 := \bigoplus_{i=1}^s Az_i'$.
Similarly, the $\C$-span is denoted $\g_0(-1)^0$.

\p\label{centfree} Consider the centraliser $\g_0^e = \{x \in \g_0 : [x,e] = 0\}$. Now $\sl_2$-theory tells us that $\ad(e) : \g_\Q(i) \rightarrow \g_\Q(i+2)$ is
surjective for all $i \geq -1$.  After inverting finitely many primes $S_3 \subseteq \Z$ in $A$ we may assume that $\ad(e) : \g_A(i) \rightarrow \g_A(i+2)$ is
also surjective for $i \geq -1$. In this case there is a basis $x_1,...,x_r, x_{r+1},...,x_m$ for the free $A$-module $\bigoplus_{i\geq 0} \g_A(i)$
such that $x_1,...,x_r$ is a basis for the free $A$-module $\g_A^e := \g_A \bigcap \g_0^e$. We may suppose that $$x_i \in \g_0(n_i)$$ for integers $n_i \in \N$.
Since $p \in \pi(A)$ is not invertible in $A$ it is a very good prime for the root system, and so it follows from the results of \cite[\S 2]{Ja04a} along with
Lemma~\ref{goodreps} that $\g_A^e \bigotimes_{A} \K_p$ is the centraliser of $e_p$ in $\g_p$.

\p For $p \in \pi(A)$ we now define
\begin{eqnarray*}
& & \m_0 := \g_0(-1)^0 \bigoplus \sum_{i<-1} \g_0(i) ;\\
& & \m_A := \g_A(-1)^0 \bigoplus \sum_{i<-1} \g_A(i) ;\\
& & \m_p := \m_A \bigotimes_A \K_p.
\end{eqnarray*}
Before we proceed we would like to record
some basic properties of these algebras.
\begin{lem}\label{factsaboutm}
Suppose that $A$ is an admissible ring containing $S_1^{-1} \cup S_2^{-1} \cup S_3^{-1}$.
The following hold for $p \in \pi(A)$:
\begin{enumerate}
\item{$\m_p$ is $[p]$-nilpotent;}
\item{$\chi$ vanishes on $[\m_0, \m_0]$ and $[\m_A, \m_A]$;}
\item{$\chi_p$ vanishes on $[\m_p, \m_p]$ and $\m_p^{[p]}$;}
\item{$\m_p \bigcap \g_p^{e_p} = 0$;}
\item{$\dim_{\K_p} (\m_p) = \dim_\C (\m_0) = d(\chi_p) = m + s - r$.}
\end{enumerate}
\end{lem}
\begin{proof}
Since each $p\in \pi(A)$ is very good for $\Phi$ we have $\z(\g_p) = 0$. Therefore $x \in \g_p(i)$ if and only if $\ad(x) : \g_p(j) \rightarrow \g_p(i + j)$ for all $j$.
Now if $x\in \g_p(i)$ then  $\ad(x^{[p]}) = \ad(x)^p : \g_A(j) \rightarrow \g_A(pi + j)$ which means that $x^{[p]} \in \g_p(pi)$.
Now part (1) follows from $\m_p \subseteq \bigoplus_{i<0} \g_p(i)$.
The same inclusion proves part (4), given $\g_p^{e_p} \subseteq \bigoplus_{i\geq 0} \g_p(i)$.

Part (2) follows from the fact that $\chi$ is supported on $\g_0(-2)$, and similar for part (3). For part (5), the
first equality is straightforward. The second follows from \cite[\S 2]{Ja04a} along with the equality
$\dim \g_0^e = \dim \g_0(0) + \dim \g_0(1)$ which may be deduced from $\sl_2$-theory. The final equality in (5) is deduced
similarly.
\end{proof}

\p By parts (2) and (3) of the above lemma we may define a module structure for $\m_0$, $\m_A$ and $\m_p$ on
$\C$, $A$ and $\K_p$ respectively by allowing these Lie algebras to act via $\chi$ or $\chi_p$. These three rings with their module structures
shall be denoted $\C_\chi$, $A_\chi$ and $\K_{\chi,p}$ with unit elements $1_\chi$, $1_\chi$ and $1_{\chi, p}$ respectively. Define
\begin{eqnarray*}
& & Q_{\chi, 0} := U(\g_0) \bigotimes_{U(\m_0)} \C_\chi ;\\
& & Q_{\chi, A} := U(\g_A) \bigotimes_{U(\m_A)} A_\chi ;\\
& & Q_{\chi, p} := U(\g_p) \bigotimes_{U(\m_p)} \K_{\chi,p}.
\end{eqnarray*}
The first of the above three is known as the ordinary Gelfand-Graev module. For indexes $(\a, \b) \in \Z^s_{\geq0} \times \Z^m_{\geq0}$ we consider the PBW monomials
$$z^\a x^\b := z_1^{a_1} \cdots z_s^{a_s}x_1^{b_1}\cdots x_m^{b_m} \in U(\g_A) \subseteq U(\g_0).$$
By the PBW theorem the images of $z^\a x^\b$ with $(\a, \b) \in \Z^s_{\geq0} \times \Z^m_{\geq0}$ form a basis for $Q_{\chi, 0}$. By abuse of notation we denote these
images by the same symbols. The enlargement of $A$ made in (\ref{symplect}) ensures that $Q_{\chi, A}$ is a free $A$-module with basis
$\{z^\a x^\b | (\a, \b) \in \Z^s_{\geq0} \times \Z^m_{\geq0}\}$. We define an important filtration on $Q_{\chi, 0}$ by setting
$$|(\a,\b)|_e := \sum_{i=1}^s a_i + \sum_{i=1}^m (n_i + 2) b_i $$
and declaring $z^\a x^\b$ to lie in degree $|(\a, \b)|_e$.

\p We now recall the definitions of the ordinary and modular finite $W$-algebras, and write the action on the right.
\begin{defn}\label{fullyadmissible}
The ordinary finite $W$-algebra is $$U(\g_0, e) := \End_{\g_0}(Q_{\chi, 0})^{\op}.$$
The modular finite $W$-algebra is $$\wU(\g_p, e_p) := \End_{\g_p}(Q_{\chi, p})^\op.$$
We regard both of these as right endomorphism rings of their respective modules.
\end{defn}
\begin{rem}\label{discussdefns}
The notation $U(\g_0,e)$ is obviously used to draw parallels with the enveloping algebra $U(\g_0)$
however the choice of notation $\wU(\g_p, e_p)$ is less clear. The use of $\wU$ follows \cite{Pr10a}
where two modular analogues of the ordinary finite $W$-algebra were studied: the first of them is the
mod $p$ reduction of a certain $A$-form on $U(\g_0,e)$ and this is denoted $U(\g_p, e_p)$,
whilst the second analogue is $\wU(\g_p,e_p)$. According to \cite[Theorem~2.1]{Pr10a} the algebra $\wU(\g_p,e_p)$ is a split central extension of
$U(\g_p,e_p)$ by a polynomial ring. In a forthcoming joint work with Simon Goodwin \cite{GT16} we shall explain
how to define the modular analogue of $U(\g_0,e)$ without any reference to $\C$-algebras,
and develop their structure theory.
\end{rem}

\p\label{definethat} Note that $Q_{\chi, 0}$ is a cyclic $U(\g_0)$-module generated by $1\bigotimes 1_\chi$. As such every element $\Theta \in U(\g_0, e)$ is determined by $(1\bigotimes 1_\chi)\Theta$.
According to Theorem~\cite[Theorem~4.6]{Pr02} there are endomorphisms $\Theta_1,...,\Theta_r \in U(\g_0, e)$ which generate $U(\g_0,e)$ as an algebra, such that
$$(1 \bigotimes 1_\chi)\Theta_i = (x_k + \sum_{0 < |(\a, \b)|_e \leq n_k + 2} \lambda_{\a, \b}^k z^\a x^\b) \bigotimes 1_\chi$$
for scalars $\lambda_{\a,\b}^k \in \Q$, which vanish if either of the following two hold:
\begin{itemize}
\item{$|(\a,\b)|_e = n_k + 2$ and $|\a| + |\b| = 1$; or}
\item{$\a = 0$ and $b_1 = \cdots = b_r = 0$.}
\end{itemize}
Furthermore {\it loc. cit.} states that there exist polynomials $F_{i,j} \in \Q[X_1,...,X_r]$
with $1\leq i,j \leq r$ such that the only relations between $\Theta_1,...,\Theta_r$ are the commutation relations
$$[\Theta_i, \Theta_j] = F_{i,j}(\Theta_1,...,\Theta_r).$$
Once again we enlarge $A$ by inverting a finite collection $S_4 \subseteq \Z$ of prime numbers, in order to ensure that the
scalars $\lambda_{\a,\b}^k$ and the coefficients of all polynomials $F_{i,j}$ lie in $A$.
\begin{defn}
We say that an admissible ring $A$ for {\sf X} is ``fully admissible'' if the primes $S_1 \bigcup S_2 \bigcup S_3\bigcup S_4\subseteq \Z$
defined above are all invertible in $A$. Note that this definition depends upon the data $(\Phi, \O, e)$.
\end{defn}
\begin{rem}\label{reeem}
\begin{enumerate}
\item{If $A$ is fully admissible and $p\in \pi(A)$ then $\m_0$, $Q_{\chi, 0}$ and $U(\g_0,e)$ all contain $A$-forms which can be reduced modulo $p$.}
\item{The number $\min\{\Pi(A) : A \textnormal{ fully admissible for } {\sf X}\}$ is well defined, however in order to estimate this quantity we would
need to obtain precise information about the structure constants of $U(\g_0, e)$ (see \cite[\S 7]{Lo10a} for a survey of known results in type {\sf A}).}
\end{enumerate}
\end{rem}

\p We now reduce the endomorphisms $\Theta_1,...,\Theta_r$ modulo $p$ to obtain endomorphisms of $Q_{\chi, p}$.
The hypotheses which have been placed on $A$ over the course of the preceding
paragraphs ensure that $\Theta_i$ preserves $Q_{\chi, A}$ for all $i$. Following \cite[2.5]{Pr10a} we define
$$\oTheta_i := \Theta_i \bigotimes_A \K_p \in \End_{\g_p}(Q_{\chi, p})^\op.$$

\p\label{centredd} Recall from Section~\ref{reas} the $p$-centre $Z_0(\g_p)$ of the restricted Lie algebra $\g_p$. There is a natural identification
$\Spec Z_0(\g_p) \leftrightarrow (\g_p^*)^{(1)}$ and for $\eta \in \g_p^\ast$ we write $I_\eta$ for the corresponding ideal.
We consider the functors
\begin{eqnarray*}
\begin{array}{ccl} \pi_\eta  : U(\g_p)\md &\longrightarrow & U_\eta(\g_p)\md;\\
M &\longmapsto & M/I_\eta M.\end{array}
\end{eqnarray*}
In \cite{Pr10a} the notation $Q_\chi^\eta$ was used to denote $\pi_\eta(Q_{\chi, p})$. We have chosen to use the operator notation
$\pi_\eta$ as we intend to prove some properties of $\pi_\eta$ in a very general setting in the next section. We warn the reader that
we shall occasionally abuse notation by using $\pi_\eta$ to denote the projection $M \twoheadrightarrow M/I_\eta M$.

Notice that
endomorphisms of a $\g$-module $M$ also preserve $I_\eta M$ and hence they descend to $\pi_\eta(M)$.
In this way we can define endomorphisms $\theta_1^\eta,...,\theta_r^\eta \in U_\eta(\g_p,e_p) := \End_{\g_p}(\pi_\eta Q_{\chi, p})^\op$ by 
writing $\overline{1}_\eta$ for the image of $(1\bigotimes 1_\chi)$ in $\pi_\eta(Q_{\chi, p})$ and setting
$$(\overline{1}_\eta)\theta_i^\eta := (1 \bigotimes 1_\chi)\oTheta_i + I_\eta Q_{\chi, p}$$
for $i=1,...,r$. The following result was proven in \cite[Lemma~2.2]{Pr10a}, and that same proof applies here (using Lemma~\ref{factsaboutm}).
\begin{lem}\label{crggm}
The following are true for $\eta \in \chi_p + \m_p^\perp$:
\begin{enumerate}
\item{$\pi_\eta(Q_{\chi, p}) \cong U_\eta(\g_p) \bigotimes_{U_\eta(\m_p)} \K_{\chi, p}$ as $\g$-modules so $$\dim \pi_\eta(Q_{\chi, p}) = p^{\dim(\m^\perp)};$$}
\item{$\pi_\eta(Q_{\chi, p})$ is a projective generator for $U_\eta(\g_p)$ and $$U_\eta(\g_p) \cong \Mat_{D(\chi_p)} U_\eta(\g_p,e_p).$$ }
\item{the monomials $(\theta_1^\eta)^{a_1} \cdots (\theta_r^\eta)^{a_r}$ with $0 \leq a_i \leq p-1$ form a $\K_p$-basis for $U_\eta(\g_p, e_p)$.}
\end{enumerate}
\end{lem}

\p We remind the reader of the notation $\m_p^\perp := \{ f\in \g_p^\ast | f(\m_p) = 0\}$. The previous lemma leads immediately to
a description of a very nice basis for $\pi_\eta(Q_{\chi, p})$ as a free $U_\eta(\g_p, e_p)$-module.
Write $[p] := \{0,1,...,p-1\}$ and let $$\G := \{ z^\a x^\b \in Q_{\chi, p} | (\a,\b) \in [p]^s \times [p]^m, b_1 = b_2 = \cdots = b_r = 0\}$$
We call the set $\G$ ``the global basis for'' $Q_{\chi, p}$. The reason for this nomenclature is explained in the 
following result, also due to Premet \cite[Lemma~2.3]{Pr10a}. The proof requires only Lemma~\ref{crggm} above.
\begin{lem}\label{universalbasis}
For each $\eta \in \chi_p + \m_p^\perp$ we let $\pi_\eta \G$ denote the image of $\G$ in $\pi_\eta(Q_{\chi, p})$. Then $\pi_\eta(Q_{\chi,p})$
is a free right $U_\eta(\g_p, e)$-module, and $\pi_\eta\G$ is a basis. $\hfill \qed$.
\end{lem}
Notice that $|\G| = m + s - r$ which equals $d(\chi_p)$ by Lemma~\ref{factsaboutm}. In other words, $\pi_\eta Q_{\chi, p}$ is free of rank
$D(\chi_p)$ over $U_\eta(\g_p, e_p)$.

\p Fix $\eta \in \chi_p + \m_p^\perp$ and order the set $\G = \{b_1,...,b_{{D(\chi)}}\}$. By the previous lemma every $v \in \pi_\eta Q_{\chi,p}$
can be written $v = \sum_{i} \pi_\eta(b_i) c_i$ for unique coefficients $c_i \in U_\eta(\g_p,e_p)$. Therefore, if $u \in U_\eta(\g_p)$ has $u \pi_\eta(b_i) = 0$
for all $i$ then $u$ annihilates $\pi_\eta Q_{\chi, p}$. Now consider the matrix isomorphism $U_\eta(\g_p) \cong \Mat_{{D(\chi_p)}} U_\eta(\g_p,e_p)$ from
part (2) of Theorem~\ref{crggm}. The column space of $\Mat_{{D(\chi_p)}} U_\eta(\g_p,e_p)$ is a left $U_\eta(\g_p)$-module $C$, with $U_\eta(\g_p) \cong \bigoplus_{D(\chi_p)} C$.
It follows that $C$ is projective and, comparing with the decomposition $U_\eta(\g_p)\cong \bigoplus_{D(\chi_p)} \pi_\eta Q_{\chi, p}$ we conclude
that $C\cong \pi_\eta Q_{\chi,p}$ (here we make use of \cite[Corollary 6.2\& Proposition 6.3]{Pi82}). It is easy to see that the only element which
annihilates the column space is zero so we have proved the following:
\begin{cor}\label{intersectann}
The left annihilator of $\pi_\eta \G$ is null: 
\begin{eqnarray*}
\bigcap_{b\in \G} \Ann_{U_\eta(\g_p)} \pi_\eta b = 0.
\end{eqnarray*}
$\hfill \qed$
\end{cor}

\section{Modular GGG modules and their central reductions}\label{generalprops}
\setcounter{parno}{0}

\p \label{bounding} The previous section explained how to start with an ordinary finite $W$-algebra and reduce into positive characteristic.
The context of the main theorem of this paper starts with a Lie algebra in sufficiently large positive characteristic, and attaches a finite $W$-algebra to
each nilpotent orbit. We begin this section by reconciling these two vantage points. In doing so, we shall define all of the objects in the
statement of the main theorem using the constructions of the previous section. This will require us to carefully restate our notation.
Pick an indecomposable, reduced, crystallographic root system $\Phi$ and let $G_0$ be the complex, simple, simply connected algebraic
group with root system $\Phi$, Lie algebra $\g_0$ and choose a $\Z$-form $\g_\Z$. For each nilpotent orbit datum {\sf X} we
write $\O_{\sf X} \subseteq \g_0$ for the corresponding orbit and choose an element $e_{\sf X} \in \O_{\sf X}$ satisfying the
properties of Lemma~\ref{goodreps}.
\begin{defn}\label{pidef}
Let $A_{\sf X}$ be the smallest fully admissible ring for $(\Phi, \O_{\sf X}, e_{\sf X})$ and write $$\Pi({\sf X}) := \Pi(A_{\sf X}).$$
\end{defn}

\p\label{pnotes} For $p \geq \Pi({\sf X})$ we write $\K := \overline{\F}_p$ and let $G$ be the simple, simply connected algebraic group over $\K$ with Lie algebra $\g$ (these were denoted
$G_p$ and $\g_p$ in the previous section). Let $\O \subseteq \g$ denote the nilpotent $G$-orbit with label {\sf X} and denote
$e := e_{\sf X} \bigotimes_\Z 1 \in \g_Z\bigotimes_\Z \K = \g$, which lies in the $G$-orbit with Bala--Carter label {\sf X} by construction.
Recall the form $(\cdot, \cdot)$ on $\g_0$ from (\ref{formandadmiss}). Since $p \in \pi(A)$ the form reduces to a non-degenerate form on $\g$, which we also denote by $(\cdot, \cdot)$.
In order to simplify notation, and bring it into line with that of the introduction and the main theorem,
we write $e$ for $e\bigotimes 1$, we write $\chi = (e,\cdot) \in \g^\ast$, we write $\m$ for the nilpotent algebra $\m_p \subseteq \g$.
Furthermore the one dimensional representation of $\m$ afforded by $\chi$ shall be denoted $\K_\chi$. With these notations in place, we remind
the reader of our main objects of study:
\begin{eqnarray*}
Q_\chi := U(\g) \bigotimes_{U(\m)} \K_\chi;\\
\wU(\g,e) := \End_{\g}(Q_\chi)^\op.
\end{eqnarray*}

\p Following \cite{Pr10a} we study the Gelfand-Graev module $Q_\chi$ via its quotients $Q_\chi/I_\eta Q_\chi$ where $I_\eta$ is a maximal ideal of the $p$-centre.
This seems to be part of an extremely general procedure and so we begin by developing some tools in a general setting. For each ideal $I$ of $Z_0(\g)$ we
may consider the central reduction $$U_I(\g) := U(\g)/ I U(\g).$$ In the special case where $I = I_\chi$ is a maximal ideal of $Z_0(\g)$ we recover the reduced
enveloping algebra $U_\chi(\g) = U_{I_\chi}(\g)$, where $I_\chi$ is the ideal defined in formula (\ref{idealIchi}) of $\S \ref{reas}$. We shall refer to $U_I(\g)$ as ``the reduced
enveloping algebra of $I$''. We write $V(I)$ for the elements $\chi \in \g^\ast$ such that $I \subseteq I_\chi$. The purpose of this section is to prove some
elementary facts about these central reductions of $U(\g)$.
\begin{lem}\label{firstlemma}
The simple modules for $U_I(\g)$ are precisely the simple $\g$-modules with $p$-character $\chi \in V(I)$.
\end{lem}
\begin{proof}
Note that a simple $U_I(\g)$-module $M$ is simple as a $U(\g)$-module and so has a $p$-character $\chi$. Then $I_\chi$ is the unique maximal ideal which kills
$M$, and since $IM = 0$ we must have $I \subseteq I_\chi$, for otherwise we would find that $1 \in I + I_\chi$ kills $M$. 
\end{proof}

\p We study the functor of ``$I$-coinvariants''
\begin{eqnarray*}
\pi_I : U(\g)\md &\longrightarrow & U_I(\g)\md ;\\
M &\longmapsto & M/I M
\end{eqnarray*}
Note that there is a dual theory of $I$-invariants which shall not interest us here. In the special case where $I =I_\chi$ is maximal we may also use the notation $\pi_\chi := \pi_{I_\chi}$.
\begin{lem}\label{homsetc}
Let $M$ be a $U(\g)$-module and $N$ be a $U_I(\g)$-module.
\begin{enumerate} 
\item{every $U(\g)$-module map $M \rightarrow N$ factors through $M \twoheadrightarrow \pi_I(M)$;}
\item{$\pi_I$ is left adjoint to the inclusion $U_I(\g)\md \rightarrow U(\g)\md$;}
\end{enumerate}
\end{lem}
\begin{proof}
If $\varphi: M \rightarrow N$ and $IN = 0$ then $\varphi(IM) = 0$. So $\varphi = \tilde\varphi \circ \pi_I$ where $\tilde\varphi : \pi_I(M) \rightarrow N; x + IM \mapsto \varphi(x)$, which proves (1). For (2) we must show that $\Hom_{U_I(\g)}(\pi_I M, N) \cong \Hom_{U(\g)}(M, N)$. We define a map $\Hom_{U_I(\g)}(\pi_I M, N) \rightarrow \Hom_{U(\g)}(M, N)$ by $x \mapsto x \circ \pi_I$. The injectivity is trivial whilst the surjectivity is part (1). 
\end{proof}

\p \label{mapsnots} Let $M,N$ be $U(\g)$-modules and suppose $\varphi : M \rightarrow N$. Then $\pi_I$ induces a map $\varphi_I : \pi_I(M) \rightarrow \pi_I(N)$. For $x\in M$ we have $$\varphi_I(x + I M) := \varphi(x) + I N$$ and once again we may use the notation $\varphi_\chi$ in the case where $I = I_\chi$ is maximal.
\begin{cor}\label{cokerandpi}
$\Coker(\varphi_I) \cong \pi_I(\Coker(\varphi))$.
\end{cor}
\begin{proof}
The claim follows from Lemma \ref{homsetc} along with the fact that left adjoints are right exact.
\end{proof}

\p \label{vanishingofmax} We shall need another elementary lemma.
\begin{lem}\label{maxvanishes}
Let $M$ be a finitely generated $U(\g)$-module. If $\pi_\eta M = 0$ for all $\eta \in \g^\ast$ then $M = 0$.
\end{lem}
\begin{proof}
We prove the contrapositive. Assume $M\neq 0$. Since $M$ is finitely generated and $U(\g)$ is Noetherian $M$ has a maximal submodule $N$, hence a simple quotient $M/N$ which must have a $p$-character $\chi$. By part 1 of Lemma~\ref{homsetc} the map $M \twoheadrightarrow M/N$ factorises as $M \twoheadrightarrow \pi_\chi(M) \twoheadrightarrow M/N \neq 0$ so that $\pi_\chi(M) \neq 0$.
\end{proof}

\p Recall that we write $\m^\perp \subseteq \g^\ast$ to denote the annihilator of $\m$ in $\g^\ast$.
\begin{lem}\label{pivanishes}
We have that $\pi_\eta Q_\chi \neq 0$ if and only if $\eta \in \chi + \m^\perp$.
\end{lem}
\begin{proof}
Lemma~\ref{crggm} shows that $\pi_\eta Q_\chi \neq 0$ for all $\eta \in \chi + \m^\perp$. Conversely if $\eta \notin \chi + \m^\perp$ then $\eta|_\m \neq \chi|_\m$. It follows that there exists an element $x\in \m$ such that $\eta(x) \neq \chi(x)$. Let $1\bigotimes 1$ denote the generator in $Q_\chi$. Recall from Lemma~\ref{factsaboutm} that $\chi(\m^{[p]}) = 0$ and so $(x^p - x^{[p]} - \eta(x)^p)(1\bigotimes 1) = (\chi(x)^p - \chi(x^{[p]}) - \eta(x)^p)(1\bigotimes 1) = (\chi(x) - \eta(x))^p(1\bigotimes 1)$. This means that $1\bigotimes 1 \in I_\eta Q_\chi$, and so $Q_\chi = I_\eta Q_\chi$ and $\pi_\eta Q_\chi = 0$.
\end{proof}

\p\label{filtrationgeneralities} To finish the section we provide a precise description of the annihilator of $Q_\chi$ in $Z_0(\g)$ and show that $Q_\chi$
is a free module over $Z_0(\g)$ modulo the annihilator. The proof shall require two facts regarding filtered rings and modules and we shall state these
facts in full generality, for the sake of clarity. In order to do so we introduce a little notation. Let $R$ be a $\Z$-filtered $\K$-algebra with $R_i = 0$ for $i < 0$
and $R_0 = \K$. Let $M = \bigcup_{i \in \Z_{\geq 0}} M_i$ be a filtered module over $R$. We shall consider the associated graded module
$\gr(M) = \bigoplus_{i\in \Z} \gr(M)^i$ where $\gr(M)^i := M_i/M_{i-1}$. For a submodule $N \subseteq M$ the filtration on $M$ induces filtrations on
both $N$ and $M/N$.

\begin{lem}\label{astepintherightdirection} The following hold:
\begin{enumerate}
\item{$\gr(M/N) \cong \gr(M)/\gr(N)$;}

\smallskip

\item{If $\gr(M)$ is a free module over $\gr(R)$ with a finite homogeneous basis, then $M$ is free over $R$ of the same rank.}
\end{enumerate}
\end{lem}
\begin{proof}
Part (1) follows directly from \cite[Proposition~7.6.13]{MR01} in view of the fact that $N_i := M_i \bigcap N$ and $(M/N)_i := M_i + N$.

We sketch the proof of the second part. Pick a homogeneous basis $m_1,...,m_l$ for $\gr(M)$ over $\gr(R)$ and suppose $m_i \in \gr(M)^{d_i}$.
Pick elements $n_1,...,n_l \in M$ with $n_i + M_{d_i - 1} = m_i$ and claim that $n_1,...,n_l$ is a basis for $M$ over $R$.
Suppose that there is a linear dependence $\sum r_i n_i = 0$ with $r_i \in R_{d'_i}\setminus R_{d'_{i-1}}$ and write $D:= \max \{ d_i + d'_i : i=1,...,l\}$.
Then $\sum r_i n_i + M_{D-1} = \sum_{d_i + d'_i = D} (r_i + R_{d'_i - 1})m_i = 0$ is a linear dependence in $\gr(M)$ between $m_1,...,m_l$
which shows that $r_1 = \cdots = r_l = 0$. The hypothesis that $\gr(M)$ is finitely generated implies that all of the filtered pieces are
finite dimensional over $\K$. Now check by induction that for $D \geq 0$ we have
$$\dim(\sum_{i=1}^l (\bigoplus_{j=0}^{D - d_i} \gr(R)^{j}) m_i) = \dim( \sum_{i=1}^l R_{D-d_i}n_i).$$
Since $m_1,...,m_l$ are a basis for $\gr(M)$ over $\gr(R)$ it follows that the left hand side of the equality is equal to $\dim( \bigoplus_{i=0}^D \gr(M)^i)$.
It follows from the general theory of filtered vector spaces that $\dim( \bigoplus_{i=0}^D \gr(M)^i) = \dim(M_D)$. We conclude that $n_1,...,n_l$ generate
$M$ over $R$, which completes the proof.
\end{proof}

\p We remind the reader that $Z_0(\g)$ identifies with regular functions on the Frobenius twist $(\g^\ast)^{(1)}$ and that the ideal of $Z_0(\g)$ consisting of functions which vanish on $\chi + \m^\perp$ is written $I(\m)$.
\begin{prop}\label{ImannihilatesQ}
\begin{enumerate}
\item{$I(\m) = \Ann_{Z_0(\g)}Q_\chi$;}
\smallskip
\item{$Q_\chi$ is a free $Z_0(\g)/I(\m)$-module of rank $p^{\dim(\m^\perp)}$.}
\end{enumerate}
\end{prop}
\begin{proof}
As a consequence of the definitions, $I(\m)$ is generated by elements $x^p - x^{[p]} - \chi(x)^p$ with $x \in \m$. Using the observations of the proof of Lemma~\ref{pivanishes} we see that these elements annihilate $Q_\chi$, and so $I(\m) \subseteq \Ann_{Z_0(\g)}Q_\chi$, which proves (1).

By part (1), $Q_\chi$ is a $Z_0(\g)/I(\m)$-module. To prove part (2) we consider $Q_\chi$ as a filtered $Z_0(\g)/I(\m)$-module, both filtrations induced from the PBW filtration on $U(\g)$. By part (1) of Lemma~\ref{astepintherightdirection} the associated graded algebra is $\gr (Z_0(\g) / I(\m)) = \gr Z_0(\g) / \gr I(\m) = S(\g)^p / (x^p : x\in \m)$. Similarly the associated graded GGG module is $\gr Q_\chi = \gr U(\g) / \gr U(\g)\m^\chi = S(\g)/(\m)$ where $U(\g) \m^\chi$ is the left ideal of $U(\g)$ generated by $x - \chi(x)$ with $x \in \m$. We have now shown that $\gr Q_\chi$ is a free $\gr Z_0(\g) / I(\m)$-module of rank $p^{\dim(\m^\perp)}$. Using part (2) of Lemma~\ref{astepintherightdirection} we see that $Q_\chi$ is a free $Z_0(\g)/I(\m)$-module of the same rank.
\end{proof}

\section{Proof of the main theorem}\label{proofofthe}
\setcounter{parno}{0}

\p\label{restateand1imp2} The first objective of this paper is to prove Theorem~\ref{mainthm}, which states that
\begin{enumerate}
\item[(1)]{$U_{I(\m)}(\g) \cong \bigoplus_{D(\chi)} Q_\chi$ as left $U(\g)$-modules;}
\item[(2)]{$U_{I(\m)}(\g) \cong \Mat_{{D(\chi)}}\wU(\g,e)$ as algebras.}
\end{enumerate}
provided the characteristic of the underlying field is large. First let us show that $(1)$ implies $(2)$. Let $\End^R_\g(M)$ be the right endomorphisms of a left $\g$-module $M$. It is a general fact that the right endomorphism ring of any ring, viewed as a left module over itself, is isomorphic to the ring in question. In our case we have $U_{I(\m)}(\g) \cong \End_{\g}^R(U_{I(\m)}(\g))$ where $U_{I(\m)}(\g)$ is viewed as a left $\g$-module by left multiplication. Using (1) and \cite[Corollary 3.4a]{Pi82} we have
$$\End_{\g}^R(U_{I(\m)}(\g)) \cong \End^R_\g(\bigoplus_{D(\chi)} Q_\chi) \cong \Mat_{{D(\chi)}} \End_\g^R(Q_\chi).$$ 
Since we use $\End_\g(M)$ to denote left endomorphisms of $M$ we have $\wU(\g,e) = \End_\g(Q_\chi)^\op \cong \End_\g^R(Q_\chi)$, which proves (2).

\p \label{outlineofproof} We now set ourselves to proving (1). We shall construct an explicit isomorphism between $U_{I(\m)}(\g)$ and $\bigoplus_{D(\chi)} Q_\chi$ in three steps:
\begin{enumerate}
\item[(i)]{first of all we define a $U(\g)$-module homomorphism from $U(\g)$ to $\bigoplus_{D(\chi)} Q_\chi$ and show that it is surjective by studying the quotients by centrally generated ideals;}
\item[(ii)]{next we observe that $I(\m) U(\g)$ is contained in the kernel, so that $$U_{I(\m)}(\g) \twoheadrightarrow \bigoplus_{D(\chi)} Q_\chi;$$}
\item[(iii)]{finally, we shall show that both $U_{I(\m)}(\g)$ and $\bigoplus_{D(\chi)} Q_\chi$ are free modules of the same rank over $Z_0(\g)/I(\m)$.}
\end{enumerate}
This shall complete the proof.

\p Recall from Lemma~\ref{universalbasis} that $Q_\chi$ contains a special subset $\G$ such that
$\pi_\eta Q_\chi$ is a free right $U_\eta(\g,e)$-module with basis $\pi_\eta \G$ for every $\eta \in \chi + \m^\perp$.
Order the set $\G = \{b_1,...,b_{D(\chi)}\}$ and for $i = 1,...,{D(\chi)}$ we let $Q_\chi[i]$ denote a copy of $Q_\chi$
with label $i$, and let $\psi_i : Q_\chi \rightarrow Q_\chi[i]$ be the identity mapping. Set $c_i := \psi_i(b_i)$. Now we define
\begin{eqnarray*}
\varphi : U(\g) & \longrightarrow & \bigoplus_{i=1}^{D(\chi)} Q_\chi[i]\\
u &\longmapsto & \sum_{i} u c_i
\end{eqnarray*}

\p We proceed with part (i) of the proof outlined in paragraph (\ref{outlineofproof}). In what follows we use the notation $\varphi_\eta$ with $\eta\in \g^\ast$ introduced in paragraph (\ref{mapsnots}). 
\begin{prop}
The map $\varphi$ is surjective.
\end{prop}
\begin{proof}
We shall show that $\Coker(\varphi) = 0$. First of all we fix $\eta \in \chi + \m^\perp$. The space $\Ker(\varphi_\eta)$ is the set
of elements $u\in U_\eta(\g)$ which simultaneously annihilate every element of $\pi_\eta \G$, and by Corollary \ref{intersectann}
this set is zero. It is well-known that $\dim U_\eta(\g) = p^{\dim(\g)}$ whilst, by part (1) of Lemma~\ref{crggm} and part (5) of
Lemma~\ref{factsaboutm}, we have $\dim \pi_\eta(\bigoplus_{D(\chi)} Q_\chi) = p^{\dim(\g)}$. This ensures that $\varphi_\eta$
is actually an isomorphism. In particular it is surjective. Now taking $\eta \notin \chi + \m^\perp$ we have
$\pi_\eta ( \bigoplus_{D(\chi)} Q_\chi) = 0$ by Lemma~\ref{pivanishes}, so that $\varphi_\eta$ is trivially surjective. What transpires
is that $\Coker(\varphi_\eta) = 0$ for all $\eta \in \g^\ast$ and, by Corollary~\ref{cokerandpi}, we see that $\pi_\eta(\Coker(\varphi)) = 0$
for all such $\eta$. Now apply Lemma~\ref{maxvanishes} to see that $\Coker(\varphi) = 0$ as desired. 
\end{proof}
\begin{rem}
It would be rather convenient if we could complete the proof of the theorem by showing that $\Ker(\varphi) = 0$ using a similar reasoning as in the proposition. Unfortunately this fails as we do not have $\pi_\eta\Ker(\varphi) = \Ker(\varphi_\eta)$ in general. This is due to the fact that taking coinvariants with respect to $I_\eta$ is right exact but not left exact (Cf. Lemma \ref{homsetc}).
\end{rem}

\p We shall complete the proof by showing that $I(\m) U(\g) = \Ker(\varphi)$. The first thing to notice is that, by Proposition~\ref{ImannihilatesQ}, the ideal $I(\m)$ kills $Q_\chi$ and so also kills each $c_i$. It follows that $I(\m) U(\g) \subseteq \Ker(\varphi)$. This constitutes (ii) in the outline of the proof (\ref{outlineofproof}). To complete the proof we shall show that $U_{I(\m)}(\g)$ and $\bigoplus_{D(\chi)} Q_\chi$ are free modules of the same rank over $Z_0(\g)/I(\m)$.
\begin{lem}
\begin{enumerate}
\item{$U_{I(\m)}(\g)$ is a free module of rank $p^{\dim(\g)}$ over $Z_0(\g)/I(\m)$;}
\item{$\bigoplus_{D(\chi)} Q_\chi$ is a free module of rank $p^{\dim(\g)}$ over $Z_0(\g)/I(\m)$.}
\end{enumerate}
\end{lem}
\begin{proof}
Part (1) follows from the well-known fact that $U(\g)$ is a free module of rank $p^{\dim(\g)}$ over $Z_0(\g)$. Part (2) follows from Proposition \ref{ImannihilatesQ} in view of the equality $p^{\dim(\m^\perp)} {D(\chi)} = p^{\dim(\g)}$ which, in turn, follows from part (5) Lemma~\ref{factsaboutm}.
\end{proof}

\p Now $\varphi$ induces a surjective map $$U_{I(\m)}(\g) \overset{\tilde{\varphi}}{\twoheadrightarrow} \bigoplus_{D(\chi)} Q_\chi$$ and $\bigoplus_{D(\chi)} Q_\chi \cong U_{I(\m)}(\g)/\Ker(\tilde{\varphi})$. However $Z_0(\g)/I(\m)$ is a polynomial algebra (being isomorphic to $S(\g)^p/(x^p - \chi(x)^p : x \in \m)$) and so if a proper quotient of any free module over $Z_0(\g)/I(\m)$ is also free, then it necessarily has smaller rank. We deduce that $\Ker(\tilde{\varphi}) = 0$ and part (1) of the main theorem follows. In paragraph (\ref{restateand1imp2}) we showed that $(1)$ implies $(2)$ and so we are done. $\hfill\qed$

\p Before we continue we shall state a corollary of the main theorem. Although the following argument is well-known, we record it here for the reader's convenience.
\begin{cor}\label{maincor}
$Q_\chi$ is a projective generator for $U_{I(\m)}(\g)$-mod.
\end{cor}
\begin{proof}
Since $U(\g)$ is Noetherian so too is the quotient $U_{I(\m)}(\g)$. A projective generator is a projective $U_{I(\m)}(\g)$-module $Q$ such that every finitely generated projective module $P$ occurs as a direct summand of $\bigoplus_N Q$ for some $N > 0$ depending on $P$. According to \cite[\S 2, No. 8, Lemma 8 (iii)]{Bo72} there is an exact sequence $F_1 \rightarrow F_2 \rightarrow P$ where $F_1, F_2$ are free modules of finite rank and so, since projectives admit no non-trivial extensions, $P$ is a direct summand of $F_2 = \bigoplus_{M} U_{I(\m)}(\g) \cong  \bigoplus_{M D(\chi)} Q_\chi$ for some $M \in \Z_{> 0}$.
\end{proof}

\section{Skryabin's equivalence}\label{Skryabins}
\setcounter{parno}{0}

\p \label{skryequiv1}With the proof of the main theorem complete we take a moment to explicitly spell out the modular version of Skryabin's equivalence,
which follows immediately from the main theorem. We define $\Phi$, {\sf X} and $\Pi(X)$ according to (\ref{bounding}).
Then we choose a prime $p \geq \Pi({\sf X})$ and let $G, \g, e, \m, Q_\chi$ and $\wU(\g,e)$ have the same meaning as in (\ref{pnotes}). 
We define the category of ``$p$-Whittaker modules'' $\Wh_p(\g, \chi)$ to be the full subcategory of $U(\g)$-mod consisting of all
modules where $I(\m)$ acts trivially. Equivalently these are the modules where $I(\m) U(\g)$ acts trivially, and so $\Wh_p(\g, \chi)$ is just
notation for $U_{I(\m)}(\g)\md$. Then the main theorem subtends the following modular version of Skryabin's equivalence \cite[Appendix]{Pr02}:
\begin{thm}\label{Skry}
There is an equivalence of categories $\wU(\g,e)\md$ and $\Wh_p(\g, \chi)$ which is expressed by the functor $$V \longmapsto Q_\chi \bigotimes_{\wU(\g,e)} V$$
and the quasi-inverse $$i W \longmapsfrom W$$
where $i \in \Mat_{{D(\chi)}} \wU(\g,e)$ denotes the idempotent which has $1$ in the $(1,1)$-entry and a zero in all other entries.
\end{thm}
\noindent The proof may be gathered from \cite[5.6 \& 5.7]{MR01}.

\p We would like to make a comparison between the category of $p$-Whittaker modules, and the category of generalised Whittaker modules. In parallel with the terminology over $\C$ we define the category $\Wh(\g,\chi)$ of ``generalised Whittaker modules'' to be the full subcategory of $U(\g)\md$ consisting of modules where $x - \chi(x)$ acts locally nilpotently for all $x\in \m$.
\begin{prop}
$\Wh_p(\g,\chi)$ is a full subcategory of $\Wh(\g,\chi)$.
\end{prop}
\begin{proof}
Fix $x\in \m$. We must show that $x - \chi(x)$ acts locally nilpotently on every $p$-Whittaker module $V$. We shall prove that, in fact, $x - \chi(x)$ acts nilpotently on $V$ and the nilpotence degree is bounded above independently of $V$. To this end, we show that $(x - \chi(x))^{p^i} \in I(\m)$ for some $i > 0$ depending on $x$. We begin by claiming that $(x - \chi(x))^{p^i} \equiv x^{[p]^i} \mod I(\m)$ for $i > 0$. The proof is by induction. The case $i = 1$ follows straight from the fact that $(x - \chi(x))^p - x^{[p]} = x^p - x^{[p]} - \chi(x)^p \in I(\m)$. Now for the inductive step we make the following calculation modulo $I(\m)$
\begin{eqnarray*}
(x-\chi(x))^{p^{i+1}} &=& ((x-\chi(x))^{p^i})^{p} \\
&\overset{(\ast)}{=}& (x^{[p]^i})^p \\
&\overset{(\ast \ast)}{=}& (x^{[p]^i})^{[p]} + \chi(x^{[p]^i})^p \\
&\overset{(\ast \ast \ast)}{=}& x^{[p]^{i+1}}
\end{eqnarray*}
where $(\ast)$ follows from the inductive hypothesis, $(\ast \ast )$ holds modulo $I(\m)$ and $(\ast \ast \ast)$ follows from part (3) of Lemma~\ref{factsaboutm}. Now the lemma follows from part (1) of the same lemma, which states that $\m^{[p]^i} = 0$ for $i$ sufficiently large.
\end{proof}

\p Before we continue we would like to record a corollary which depends upon the observations of the current section. The isomorphism $U_{I(\m)}(\g) \cong \Mat_{{D(\chi)}} \wU(\g,e)$ subtends a decomposition $$U_{I(\m)}(\g) \cong \bigoplus_{D(\chi)} C$$ where $C$ denotes the column space of $\Mat_{{D(\chi)}} \wU(\g,e)$.
\begin{cor}\label{smellycor}
$C \cong Q_\chi$ as left $U_{I(\m)}(\g)$-modules. In particular $Q_\chi$ is a free right $\wU(\g,e)$-module of rank $D(\chi)$.
\end{cor}
\begin{proof}
We adopt the notation of Theorem \ref{Skry}. The column space is $C = U_{I(\m)}(\g) i$. By Skryabin's equivalence we have $C \cong Q_\chi \bigotimes_{\wU(\g,e)} i C$, however $i C = i U_{I(\m)}(\g) i \cong \wU(\g,e)$ as a $\wU(\g,e)$-bimodule. Therefore $Q_\chi \bigotimes_{\wU(\g,e)} i C \cong Q_\chi$ as required. To see the final remark it suffices to observe that, quite generally, for any ring $R$, the right endomorphisms of the column space of $\Mat_{D(\chi)} (R)$ are isomorphic to $R$ acting by right multiplication, and that the column space is a free module for this action of rank $D(\chi)$
\end{proof}

\section{The $p$-characters of $\wU(\g,e)$-modules}\label{pcharofmods}
\setcounter{parno}{0}

\p\label{skryabinsequiv} We now recall an important theorem regarding the ordinary finite $W$-algebra. Recall from (\ref{bounding}) that $G_0$
is a simple, simply connected complex group and let $e$ be some nilpotent element of $\g_0$. There is a non-degenerate $G_0$-equivariant
isomorphism $\g_0 \leftrightarrow \g^\ast_0$ and suppose $e$ identifies with $\chi$. The subalgebra $\m_0$ and the module $Q_{\chi, 0}$ was
constructed in Section~\ref{redmodp}, and $U(\g_0, e)$ is the opposite endomorphism ring of $Q_{\chi, 0}$. Skryabin's
equivalence asserts that
\begin{eqnarray*}
U(\g_0, e)\md &\longrightarrow& \Wh(\g_0, \chi) \\
V &\longmapsto & Q_{\chi, 0} \bigotimes_{U(\g_0, e)} V
\end{eqnarray*}
is part of a quasi-inverse equivalence of categories. Here $\Wh(\g_0, \chi)$ has the same meaning as per the previous section: it is the full subcategory of $U(\g_0)\md$ consisting of all modules upon which $x - \chi(x)$ acts locally nilpotently for all $x\in \m_0$. For each simple $U(\g_0, e)$-module $V$ the $U(\g_0)$-module $Q_{\chi, 0} \bigotimes_{U(\g_0, e)} V$ is also simple, and the annihilator is denoted $I_V$. The associated variety of a primitive ideal $I$ is the vanishing locus in $\g_0$ of the graded ideal $\gr I \unlhd S(\g_0)$ after the identification $S(\g_0) \leftrightarrow \C[\g_0]$ (see \cite[\S 9]{Ja04a} for a survey). The first part of the following was proven by Premet, whilst the second was proven independently by Ginzburg, Losev, Premet at roughly the same time \cite{Gi09}, \cite{Lo10b}, \cite{Pr10a}.
\begin{thm}\label{PGL}
For all simple $U(\g_0, e)$-modules $V$ the associated variety of $I_V$ contains $\Ad(G_0)e$. Furthermore all primitive ideals with associated variety containing $\overline{\Ad(G_0)e}$ are obtained in this manner.
\end{thm}

\p\label{remarksonpchars} Return to the notation $G$, $\g$, $e$, $Q_\chi$, etc, instantiated in (\ref{pnotes}). Our goal in this section is to show that the simple $U(\g)$-modules $$\{Q_\chi \bigotimes_{\wU(\g,e)} V : V \textnormal{ a simple } \wU(\g,e)\textnormal{-module}\}$$ can be distinguished by an invariant, in parallel with the above theorem. The above set coincides with the set of simple $U_{I(\m)}(\g)$-modules and by Lemma \ref{firstlemma} these are precisely the simple $U(\g)$-modules with $p$-character $\eta \in \chi + \m^\perp$.

\p We recall the definition of a decomposition class in $\g^\ast$. These are the equivalence classes under the following equivalence relation:
we say $\chi \sim \psi$ if there exists $g \in G$ such that $\Ad(g)\g_{\chi_s} = \g_{\psi_s}$ and $\Ad^*(g)\chi_n = \psi_n$ where $\chi = \chi_s + \chi_n$
and $\psi = \psi_s + \psi_n$ are the Jordan decompositions. Note that the classes are stable under scalar multiplcation. We direct the reader to \cite{PS15} for a
good overview of the theory of decomposition classes in positive characteristics. The rest of the section shall be spent proving the following theorem.
\begin{thm}\label{invariantsofsimples}
Continue to assume that $p \geq \Pi({\sf X})$ so that Theorem~\ref{Skry} holds. Write $\O^\ast = \Ad^\ast(G)\chi$. For all simple $\wU(\g, e)$-modules $V$ the closure of the decomposition class $\D(\eta)$ of the $p$-character $\eta$ of $Q_\chi \bigotimes_{\wU(\g,e)} V$ contains $\O^*$. Furthermore all simple $U(\g)$-modules with $\O^*$ contained in the closure of the decomposition class of the $p$-character are obtained by twisting simple modules of the form $Q_\chi \bigotimes_{\wU(\g,e)} V$ by elements of $G$.
\end{thm}
\begin{rem}
The theorem is phrased so as to to draw parallels between characteristic zero (Theorem \ref{PGL}) and characteristic $p > 0$. We are proposing that the closure of the decomposition class of the $p$-character might play the role of the associated variety in this theory. Notice that
\begin{enumerate}
\item{since all simple $U(\g)$-modules are finite dimensional they lie in one to one correspondence with their annihilators;}
\item{the primitive ideals in $U(\g_0)$ are $G_0$-stable.}
\end{enumerate}
Taken together, these observations ensure that the second claim of the above theorem is a precise analogue of the second claim in Theorem~\ref{PGL}.
\end{rem}

\p Given the remarks of (\ref{remarksonpchars}) the above theorem can be seen as a claim about the set $\chi + \m^\perp$ or, equivalently,
about the set $e + \m^{\perp_\g}$ where $\m^{\perp_\g} = \{x\in \g : (\m, x) = 0\}$. There is an obvious definition for a decomposition class in $\g$
parallel to the definition in $\g^*$ given above (this was the original context for the study of decomposition classes). Using the form $(\cdot, \cdot)$
we obtain a bijection $\g^* \rightarrow \g$ which preserves decomposition classes. Now Theorem~\ref{invariantsofsimples} will follow immediately 
from the proposition:
\begin{prop}\label{reduprop}
The set $\Ad(G)(e + \m^{\perp_\g})$ is the union of all decomposition classes of $\g$ containing $e$ in their closure.
\end{prop}

\p Before we begin the proof of the proposition we introduce some tools and record some observations which shall be needed in the proof. 
Recall from (\ref{qsnas}) that we constructed a $\Z$-grading on the complex Lie algebra $\g_0 = \bigoplus_{i\in \Z} \g_0(i)$ and
the definition of a fully admissible ring ensures that the $\g_A(i)$ are free $A$-modules. This means that our $\K$-Lie algebra is graded $\g = \bigoplus_{i\in \Z} \g(i)$ and,
by construction, $\bigoplus_{i<-1} \g(i) \subseteq \m$. 
\begin{lem}\label{mperp}
$\bigoplus_{i\leq 0} \g(i) \subseteq \m^{\perp_\g} \subseteq \bigoplus_{i \leq 1} \g(i)$
\end{lem}
\begin{proof}
The form $\kappa_0$ on $\g_0$ places $\g_0(i)$ and $\g_0(-i)$ in a non-degenerate pairing. Since the form $(\cdot, \cdot)$
on $\g$ is obtained from a mod $p$ reduction of a rescaling of $\kappa_0$ it follows that $(\cdot, \cdot)$ places $\g(i)$
and $\g(-i)$ in non-degenerate pairing. Now it follows that $(\bigoplus_{i\leq-j} \g(i))^{\perp_\g} = \bigoplus_{i \leq j-1} \g(i)$ for $j \geq 0$
and so the lemma follows from $\bigoplus_{i\leq-2} \g(i) \subseteq \m \subseteq \bigoplus_{i\leq -1}\g(i)$.
\end{proof}

\p We define a $\K^\times$-action on $\g$ by letting
\begin{eqnarray*}
\rho(t) &: & \g(i) \longrightarrow \g(i)\\
 & & x \longmapsto t^{2-i} x
\end{eqnarray*}
and extend to $\g$ by linearity. The previous lemma shows that $\rho(\K^\times)$ preserves $e + \m^{\perp_\g}$ and restricts
to a contracting action with unique fixed point $e$, which means
\begin{eqnarray}\label{ein}
e \in \overline{\rho(\K^\times)(e + x)}
\end{eqnarray}
for all $x\in \m^{\perp_\g}$. The contracting action is also denoted $$\rho : \K^\times \rightarrow GL(e + \m^{\perp_\g}).$$
\begin{lem}\label{contractingaction}
$\rho(\K^\times)$ preserves the decomposition classes of $\g$.
\end{lem}
\begin{proof}
Define a $\K^\times$-action on $\g$ by $\mu(t) x := t^{-i} x$ for $x \in \g(i)$. Since $\rho(t) = t^2\mu(t)$ and since the map
$x \mapsto t^2 x$ preserves the decomposition classes of $\g$ it suffices to show that the classes are $\mu(\K^\times)$-stable.
Since $\g = \bigoplus_{i\in \Z} \g(i)$ is a grading it follows that $\mu : \K^\times \rightarrow \Aut(G)$.
The main result of \cite{St61} shows that the image of $\mu$ lies in $\Ad(G)$, however the decomposition classes are clearly $\Ad(G)$-stable,
which completes the proof.
\end{proof}

\p We are now ready to give a proof of Proposition \ref{reduprop}.
\begin{proof}
First of all suppose that $\D$ is a decomposition class of $\g$ with $\D \bigcap (e + \mpg) \neq \emptyset$.
Then by (\ref{ein}) for $e + x \in \D \bigcap (e + \mpg)$ we have $e \in \overline{\rho(\K^\times)(e + x)}$.
Now since $\rho(\K^\times)$ preserves the decomposition classes (Lemma~\ref{contractingaction}) and these classes are $G$-stable, it follows that
$\Ad(G) e \subseteq \overline{\D}$.

To complete the proof we show that if $\Ad(G) e \subseteq \overline{\D}$ then $\D \bigcap (e + \mpg) \neq \emptyset$.
By assumption have $\overline{\D} \bigcap (e + \mpg) \neq \emptyset$ and so it will suffice to show that the adjoint
mapping $$G \times (\overline{\D} \bigcap (e + \mpg)) \longrightarrow \overline{\D}$$ is dominant. Indeed, if we
suppose that $\P \subseteq \overline{\D}$ is a non-empty open set contained in $\Ad(G) (\overline{\D} \bigcap (e + \mpg))$
then $\D \bigcap \P$ is also open and non-empty in $\overline{\D}$, since decomposition classes are locally closed,
and every element of $\D \bigcap \P$ is conjugate to an element of $e + \mpg$. To prove dominance it will suffice
to show that the image of $G \times (e + \mpg) \overset{\varphi}{\rightarrow} \g$ contains an open neighborhood of
$e + \mpg$, for then this neighborhood will intersect non-trivially with $\overline{\D}$ and the intersection will
be contained in $\Ad(G) (\overline{\D} \bigcap (e + \mpg))$. Such an open neighborhood of $e + \mpg$ will exist if
the differentials $d_{(1, e+ x)} \varphi : \g \times \mpg \rightarrow \g$ are shown to be surjective for all $x \in \mpg$.
A routine calculation shows that
$$d_{(1, e+ x)} \varphi(u,v) = [u, e + x] + v$$
and now we must check that $\g = [\g, e + x] + \mpg$ for $x \in \mpg$.

In (\ref{centfree}) we observed that the modular reduction of the complex centraliser of $e$ is the modular
centraliser $\g^e$. This shows that $\g^e := \{z \in \g: [z,e] = 0\} \subseteq \bigoplus_{i\geq 0} \g(i)$.
This can be rephrased as saying $\ad(e) : \g(i) \rightarrow \g(i+ 2)$ is injective for $i \leq -1$. By a standard argument
using the form $(\cdot, \cdot)$ this is equivalent to saying that $\ad(e) : \g(i) \rightarrow \g(i+ 2)$ is surjective for
$i \geq -1$ (see \cite[Theorem~1.3]{EK05} for example).

Lemma~\ref{mperp} shows that $\bigoplus_{i\leq 0} \g(i) \subseteq [e + x, \g] + \mpg$. We check that $\g(i) \subseteq [e + x, \g] + \mpg$
for $i > 0$ by induction. Suppose that $i > 0$ and we have shown that $\bigoplus_{j=1}^{i-1} \g(j) \subseteq [e + x, \g] + \mpg$ already.
Then since $\ad(e) : \g(i-2) \twoheadrightarrow \g(i)$ and $x \in \mpg \subseteq \bigoplus_{j\leq 1}\g(j)$ we have
\begin{eqnarray*}
\ad(e + x)\g(i-2) = \{[e, y] + [x, y] : y \in \g(i-2)\} \equiv \g(i) \mod \bigoplus_{j\leq i-1} \g(j)
\end{eqnarray*}
which shows that $\g(i) \subseteq [e + x, \g] + \mpg$ for all $i \in \Z$.

\end{proof}

\section{A generalisation of the main theorem}\label{directedsystem}
\setcounter{parno}{0}

\p In this final section we shall describe a whole family of Morita theorems, parameterised by the ideals $I\unlhd Z_{0}(\g)$ containing $I(\m)$. Consider the composition $$Z_0(\g) \hookrightarrow U(\g) \twoheadrightarrow U_{I(\m)}(\g).$$
\begin{lem}
The image of the $p$-centre under the above composition is isomorphic to $Z_{I(\m)}(\g) := Z_0(\g)/I(\m)$.
\end{lem}
\begin{proof}
The content of the claim is $I(\m) U(\g) \bigcap Z_0(\g) = I(\m)$. Recall that $U(\g)$ is a free $Z_0(\g)$-module, hence faithfully flat. Now the claim follows from \cite[\S 3, No. 5, Proposition 8]{Bo72}.
\end{proof}
\noindent The image of the composition should be regarded as the $p$-centre of $U_{I(\m)}(\g)$. We identify it with $Z_0(\g)/I(\m)$ and denote by $\rho : Z_0(\g) \rightarrow Z_{I(\m)}(\g)$ the natural projection. The ideals of $Z_{I(\m)}(\g)$ identify with the ideals of $Z_0(\g)$ containing $I(\m)$. Now for every $I \unlhd Z_{I(\m)}(\g)$ we have an algebra isomorphism $U_{\rho^{-1} I}(\g) := U(\g)/ \rho^{-1}(I) U(\g) \cong U_{I(\m)}(\g)/IU_{I(\m)}(\g)$ and the main theorem of this paper implies that the left regular representation of this algebra decomposes as $\bigoplus_{D(\chi)} \pi_I Q_\chi$, where $\pi_I : U_{I(\m)}(\g)\md \rightarrow U_{I}(\g)\md; V \rightarrow V/IV$. Following the same argument as was used in (\ref{restateand1imp2}) we conclude that $$U_{\rho^{-1} I}(\g) \cong \Mat_{D(\chi)} \End_\g(\pi_I Q_\chi)^\op.$$

\p In order to sharpen the previous observation we shall identify $\End_\g(\pi_I Q_\chi)^\op$ with a quotient of $\wU(\g,e)$.
We fix an ideal $I \unlhd Z_{I(\m)}(\g)$. The centre $Z(U_{I(\m)}(\g))$ acts on $Q_\chi$ and we let
$\psi : Z(U_{I(\m)}(\g)) \rightarrow \End_\K(Q_\chi)$ denote the representation. The image $\psi(I)$ clearly lies in
$\End_\g(Q_\chi)$, and by definition it must lie in the centre of the latter. The map $\End_\g(Q_\chi) \rightarrow
\End_\g(Q_\chi)^\op = \wU(\g,e)$, which identifies the two algebras as sets, is an anti-isomorphism. It restricts to an
isomorphism between the centres of the two, and so $\psi(I)$ identifies with a (non-unital) subalgebra of $\wU(\g,e)$.
Certainly $\psi(I)$ is not an ideal of $\wU(\g,e)$ but we may consider the ideal $\wI$ generated by $\psi(I)$. Write
$$\wU_I(\g,e) := \wU(\g,e)/\wI.$$
\begin{thm}\label{generalise}
Retain the hypotheses of the main theorem. Then for every ideal $I \unlhd Z_{I(\m)}(\g)$ we have
$$U_{\rho^{-1}I}(\g) \cong \Mat_{D(\chi)} \wU_I(\g,e).$$
\end{thm}
\begin{proof}
It is easy to see that if $R$ is any associative ring and $D \in \N$ then the diagonal embedding
$\Xi: Z(R) \rightarrow Z(\Mat_D(R))$ is an isomorphism. Furthermore, if $J \unlhd Z(R)$ then
we have $$\Mat_D(R/JR) \cong \Mat_D(R)/\Xi(J) \Mat_D(R).$$ If we apply this observation to the
case $R = \wU(\g,e)$ and $D = D(\chi)$, the current theorem will follow from the claim that
$\psi$ and $\Xi$ are inverse isomorphisms
$$Z(U_{I(\m)}(\g)) \underset{\Xi}{\overset{\psi}{\rightleftarrows}} Z\wU(\g,e).$$
The fact that $\psi \circ \Xi = \Id_{Z\wU(\g,e)}$ follows from the definitions.
\end{proof}

\p Recall the reduced finite $W$-algebras $U_\eta(\g,e)$ from Section \ref{generalprops}.
An alternative description of these algebras follows from the proof of the theorem.
\begin{cor}
For $\eta \in \chi + \m^\perp$ we have $U_\eta(\g,e) \cong \wU_{I_\eta}(\g,e)$. $\hfill \qed$
\end{cor}

\p The above corollary combines with the main theorem to subtend the following diagram, which
clarifies the relationship between the restricted and unrestricted modular finite $W$-algebras.
\begin{cor}\label{comdiag}
We have a commutative diagram

\begin{center}
\begin{tikzpicture}[node distance=3.2cm, auto]
 \node (A) {$U(\g)$};
\begin{scope}[node distance=1.3cm and 10cm]
 \node (B) [below of= A] {$U_{I(\m)}(\g)$};
\end{scope}
 \node (C) [right of= B] {$\Mat_{D(\chi)} \wU(\g,e)$};
  \node (D) [right of=C]{$\wU(\g, e)$};
	\begin{scope}[node distance=1.5cm and 10cm]
	\node (E) [below of= B] {$U_\chi(\g)$};
	\end{scope}
 \node (F) [right of= E] {$\Mat_{D(\chi)} U^{[p]}(\g,e)$};
  \node (G) [right of=F]{$U^{[p]}(\g, e)$};
 \draw[>=stealth', ->>] (A) to node {$ $} (C);
 \draw[>=stealth', ->>] (A) to node {$ $} (B);
 \draw[>=stealth', ->>] (B) to node {$ $} (E);
 \draw[>=stealth', ->>] (C) to node {$ $} (F);
 \draw[>=stealth', ->>] (D) to node {$ $} (G);
 \draw[>=stealth', ->] (B) to node {$\sim$} (C);
 \draw[>=stealth', ->] (E) to node {$\sim$} (F);
 \draw[left hook-latex] (D) to node {$ $} (C);
 \draw[left hook-latex] (G) to node {$ $} (F);
\end{tikzpicture}
\end{center}
where the inclusions on the right are the diagonal embeddings, the projection from $\Mat_{D(\chi)} \wU(\g,e)$ to $\Mat_{D(\chi)} U^{[p]}(\g,e)$ is induced from the natural projection from $U_{I(\m)}(\g)$ to $U_\chi(\g)$ and the projection on the far right is obtained by taking $\chi = \eta$ in the previous corollary.
$\hfill \qed$
\end{cor}

\end{document}